\documentclass[12pt]{amsart}
\usepackage{amssymb,latexsym,eucal}

\usepackage{hyperref}
\usepackage[letterpaper,centering,text={6.5in,9in}]{geometry}
\usepackage{graphicx}
\usepackage{breqn}
\usepackage{latexsym}
\usepackage[utf8]{inputenc}
\usepackage{amsmath}
\usepackage{amsthm}
\usepackage{bigints}
\usepackage{amsfonts}
\usepackage{amssymb}
\usepackage{epsfig}
\usepackage{nicefrac}
\usepackage[all]{xy}
\usepackage{graphicx}
\usepackage{xcolor}
\usepackage{enumerate}
\usepackage{mathtools}    
\DeclarePairedDelimiterX\setc[2]{\{}{\}}{\,#1 \;\delimsize\vert\; #2\,}

\newtheorem{theorem}{Theorem}[section]

\newtheorem{proposition}[theorem]{Proposition}
\newtheorem{lemma}[theorem]{Lemma}
\newtheorem{claim}[theorem]{Claim}
\newtheorem{question}[theorem]{Question}
\newtheorem{conjecture}[theorem]{Conjecture}
\theoremstyle{definition}

\newtheorem{remark}[theorem]{Remark}

\newtheorem{problem}[theorem]{Problem}
\newtheorem{definition}[theorem]{Definition}

\newcommand{\eqdef}{\;{:=}\;}

\newcommand\Symp{\operatorname{Symp}}

\def\bigmid{\ \rule[-3.5mm]{0.1mm}{9mm}\ }
\newcommand{\II}{{\mathbb I}}
\newcommand{\CC}{{\mathbb C}}

\newcommand{\RR}{{\mathbb R}}

\newcommand{\V}{{\mathrm {Volume}}}
\newcommand{\mw}{{{\mathrm {M}}}}

\newcommand{\lmw}{{\mathrm {M}}_{\mathrm{Sp}}}
\newcommand{\nlmw}{\mathrm{M}_{{Symp}}}

\begin{document}

\title[On the symplectic capacity and mean width of convex bodies]{On the symplectic capacity and mean width of convex bodies}
\author{Jonghyeon Ahn and Ely Kerman}
\address{Department of Mathematics,
University of Illinois Urbana-Champaign\\Urbana, IL, 61801, USA.}
\email{ja34@illinois.edu, ekerman@illinois.edu}
\thanks{This research was supported by a grant from the Simons Foundation and funds from the Campus Research Board of the University of Illinois Urbana-Champaign.}

\date{\today}

\begin{abstract}
In this note we consider two topics involving the relationship between the symplectic capacity and the mean width of convex bodies in $\RR^{2n}$. We first describe an alternative path from the symplectic Brunn-Minkowski inequality of Artstein-Avidan and Ostrover in \cite{ao} to another inequality from \cite {ao} that relates the capacity and mean width of convex bodies. This new path is less direct but it relates these inequalities to the quermassintegrals of convex bodies and to the local version of Viterbo's conjecture established in \cite{ab} for domains sufficiently close to the ball. 

We then consider the problem of identifying convex bodies whose mean width cannot be decreased by natural classes of symplectomorphisms. Motivated by the results in \cite{ak}, we state a conjectured characterization of convex bodies whose mean width is already minimal among all their symplectic images. To test this conjecture we identify a simple class of quadratic convex bodies whose mean width can not be decreased by linear symplectic maps near the identity. We then identify a subset of these examples that fail to satisfy the toric conditions of the conjecture, and show that one can find a nonlinear symplectomorphism that decreases their mean width. 
\end{abstract}

\maketitle

\section{Introduction}   

 Consider the vector space $\RR^{2n}$ equipped with its standard symplectic structure and let $\Symp$ denote the group of symplectomorphisms of $\RR^{2n}$. For a subset $X$ of $\RR^{2n}$, there are several equalities and inequalities that relate symplectic invariants of $X$ to classical geometric measurements optimized over the images of $X$ under $\Symp$. For embedding capacities these relationships can be straightforward. Consider, for example, the Gromov width $c_B$, which we normalize so that the width of the closed unit ball, $B^{2n}$, is one. By the {\em Extension after Restriction Principle}, \cite{eh1,schl}, we have
\begin{align*}
    c_B(X) &= \sup \{ r^2 \mid \text{ $\exists$ a symplectic embedding } r B^{2n} \hookrightarrow X \} \\
    &= \left(\sup_{\phi \in \Symp} \left\{\mathrm{inradius}(\phi(X))\right\}\right)^2.
\end{align*}
In a similar way, if $X$ is star-shaped, then for the outer ball embedding capacity $c^B$ we have 
\begin{align*}
    c^B(X) &= \inf \{ r^2 \mid \text{ $\exists$ a symplectic embedding }  X \hookrightarrow RB^{2n}\} \\
    &= \left(\inf_{\phi \in \Symp} \left\{\mathrm{circumradius}(\phi(X))\right\}\right)^2
\end{align*}
 
Relationships to optimized geometric measurements are more subtle for symplectic invariants that are defined via Hamiltonian dynamics such as the first Ekeland-Hofer capacity, $c_{EH},$ from \cite{eh1}. The most studied relationship of this type is that between the capacity and the volume of convex bodies (for which the optimization aspect is trivial). In \cite{vit}, Viterbo asked whether the inequality 
\begin{align}\label{vit}
    c_{EH}(K) \leq \left(\frac{\V(K)}{\V(B^{2n})}\right)^{\frac{1}{n}}.
\end{align}
holds for all convex bodies $K$. This question, and its implications, has motivated a large amount of research and has lead to many remarkable results. The conjectured inequality was only recently disproved by Haim-Kislev and Ostrover in \cite{hko}.

Another relationship of this type involves the capacity $c_{EH}$ and the mean width. Recall that the mean width of a convex body $K \subset \RR^{2n}$ is defined as
\begin{equation*}
    \mw (K) =\int_{S^{2n-1}}( h_K(u) +  h_K(-u))\, \, d \sigma,
\end{equation*}
where
\begin{equation*}
     h_K (u) =\sup_{k \in K} \, \langle k,u \rangle
\end{equation*}
is the support function of $K$, and $\sigma$ is the rotationally invariant probability measure on $S^{2n-1}$. Unlike the volume, the mean width is not a symplectic invariant. The following relationship between $c_{EH}$ and $\mw$ was established by Artstein-Avidan and Ostrover in \cite{ao}. 
\begin{theorem}[Artstein-Avidan and Ostrover, \cite{ao}]\label{thm:ao}
    For every convex body $K \subset \RR^{2n}$, 
\begin{equation}\label{ao}
    c_{EH}(K) \leq \frac{\mw(K)^2}{4}.
\end{equation}
\end{theorem}
In \cite{ak}, the current authors consider refinements of \eqref{ao} of the form 
\begin{equation*}
    c_{EH}(K) \leq \inf_{\phi \in \mathcal{C}}\frac{\mw(\phi(K))^2}{4}.
\end{equation*}
for natural classes of maps $\mathcal{C} \subset \Symp$. A primary question in this direction is to characterize the convex bodies that satisfy
\begin{equation}\label{eq:min}\mw(K) = \inf_{\phi \in \mathcal{C}} \mw(\phi(K)).\end{equation}

In this note, we explore two parts of this story in more detail. In Section \ref{newpath}, we revisit the proof of inequality \eqref{ao} from \cite{ao} and give an alternative argument for the last step. Although much less direct, this new argument relates the problem to the local version of Viterbo's conjecture established in \cite{ab}, as well as to other classical measurements of convex bodies, the quermassintegrals, that lie between the mean width and the volume (see Remark \ref{qm}). 

In Section \ref{nontoric}, we explore the primary theme from \cite{ak}, that toric symmetry is a preferred feature of convex bodies that satisfy \eqref{eq:min}. It is shown in \cite{ak} that if $K$ is toric, up to the action of $\mathrm{U}(n)$, then the identity matrix is a local minimum of the function $\mathrm{Sp}(2n) \to \RR$ defined by $P \mapsto \mw(PK)$.
It is also observed that being toric is not a necessary condition for this. In Section \ref{nontoric} we state, motivate and test a conjecture which asserts that toric symmetry, up to the action of $\mathrm{U}(n)$, is both a necessary and sufficient condition when one passes to the nonlinear optimization problem. To test this, we identify a large class of minimizers of the linear problem that, according to the conjecture, should not be minimizers of the nonlinear problem. We verify this assertion in some cases, and propose a general approach to verification.

\section{Revisiting the mean width-capacity inequality of Artstien-Avidan and Ostrover} \label{newpath}
  
The key to the proof of Theorem \ref{thm:ao} is the following result from \cite {ao} which expresses the Ekeland-Hofer capacity $c_{EH}(K)$ in terms of the support function $ h_K$ that also defines $\mw(K)$. 
\begin{proposition}[see Proposition 2.1 of \cite{ao}]\label{p} If the boundary of $K$ is smooth, then for any $p>1$,
\begin{equation*}
   c_{EH}(K)^{\frac{1}{2}} = \frac{1}{2} \min_{z \in \mathcal{E}_p}  \int_0^{2 \pi}  h_K(\dot{z}) \, dt
\end{equation*}
    where 
\begin{equation*}
    \mathcal{E}_p =\left\{z \in W^{1,p}(S^1, \RR^{2n}) \mid  \int_0^{2 \pi} z(t)\,dt =0,\, \int_0^{2 \pi} \langle Jz(t), \dot{z}(t) \rangle \, dt =2 \right\}.
\end{equation*}    
\end{proposition}

Together with the additive property of support functions, $h_{K_1 +K_2} = h_{K_1} +h_{K_2}$, Proposition \ref{p} immediately yields the following symplectic analogue of the Brunn-Minkowski inquality.
\begin{theorem}[Artstein-Avidan and Ostrover, \cite{ao}]
    For any two convex bodies  $K_1, K_2 \subset \RR^{2n}$, 
\begin{align}\label{bm}
    (c_{EH}(K_1+K_2))^{\frac{1}{2}} \geq (c_{EH}(K_1))^{\frac{1}{2}}+(c_{EH}(K_2))^{\frac{1}{2}}.
\end{align}
\end{theorem}
Inequality \eqref{ao} is then  derived from inequality \eqref{bm} by averaging it over the images of $K$ under all unitary matrices.  

\begin{remark}
    More generally, the argument from \cite{ao} implies that any capacity that satisfies the symplectic Brunn-Minkowski inequality \eqref{bm}, also satisfies \eqref{ao}. As noted by Ostrover, this can be used to show that certain capacities do not satisfy the symplectic Brunn-Minkowski inequality, \cite{kl1}.
\end{remark}

We now describe a different, less direct, way to derive inequality \eqref{ao} from inequality \eqref{bm}. This argument makes use of a local version of Viterbo's conjecture, recently established in \cite{ab}, as well as Steiner's formula for the volume of Minkowski sums with balls. The relevant result from \cite{ab} is the following.

\begin{theorem}[Abbondandolo and Benedetti, \cite{ab}]\label{thm:ab}
  There is a $C^3$-neighborhood $\mathcal{B}$ of $B^{2n}$ in the space of convex bodies with $C^{\infty}$-smooth boundary such that 
  \begin{align*}
       c_{EH}(K)^{\frac{1}{2}} \leq \left(\frac{\V(K)}{\V(B^{2n})}\right)^{\frac{1}{2n}},\quad \forall K \in \mathcal{B}.
  \end{align*}
\end{theorem}

\begin{remark}
To suit our present purposes, we have rephrased the original statement of Corollary 1 from \cite{ab} in terms of closed convex bodies with smooth boundary. The additional part of the statement from \cite{ab}, that covers the case of equality, has been omitted.
\end{remark}

Since the capacity $c_{EH}$ is continuous with respect to the Hausdorff metric, it suffices to prove Theorem \ref{thm:ao} for convex bodies $K \subset \RR^{2n}$ with $C^{\infty}$-smooth boundary. Fixing a convex body $K \subset \RR^{2n}$ whose boundary is $C^{\infty}$-smooth, we define a deformation of the ball $B^{2n}$ by, $$t \mapsto B^{2n} + tK.$$ The boundary of $B^{2n} + tK$ is $C^{\infty}$-smooth for all $t\geq 0.$ This is not an obvious fact, as the relationship between the smoothness of the boundary of a Minkowski sum of convex bodies, and the smoothness of the boundaries of the summands, is known to be surprisingly delicate (see, for example, \cite{kis}). The fact that one of our summands is a ball precludes these issues. Theorem 1.1 of \cite{bj} implies that the boundary of $B^{2n} + tK$ is as smooth as that of $K$,

For $\epsilon>0$, consider the function $F \colon [0,\epsilon] \to \RR$ defined by
\begin{equation*} F(t) = \frac{c_{EH}(B^{2n} + tK)^{\frac{1}{2}}}{\left(\frac{\V(B^{2n} + tK)}{\V(B^{2n})}\right)^{\frac{1}{2n}}}.\end{equation*} 
\begin{lemma}\label{F} The function $F$ has the following properties:
\begin{itemize}
    \item[(F1)] $F(0)=1$,
    \item[(F2)] $F$ is continuous,
    \item[(F3)] for all sufficiently small $\epsilon > 0$,  $F(t) \leq 1 \text{  for all  } t \in [0,\epsilon].$   
\end{itemize}
\end{lemma}
\begin{proof}
    The first two properties are clear, and, since the boundary of $B^{2n} + tK$ is $C^{\infty}$-smooth, the third follows immediately from Theorem \ref{thm:ab}.
\end{proof}

Inequality \eqref{bm} allows one to control the numerator of $F$. To manage the denominator we use Steiner's formula which, in this setting, looks like
\begin{equation}\label{steiner}
    \V(B^{2n}+tK) = \sum_{i=0}^{2n} \binom{2n}{i} W_i(K) t^{2n-i}.
\end{equation}
Here, $W_i(K)$ is the $i$th quermassintegral of $K$ and we note that $W_0(K) = \V(K)$ and \begin{equation}\label{quer}
    W_{2n-1}(K) = \frac{\V(B^{2n})}{2} \mw(K) 
\end{equation} (see \S 7.2 of \cite{schn}). Now consider the function $\widetilde{F}\colon [0,\, \infty) \to \RR$ defined by
\begin{equation*} \widetilde{F}(t) \eqdef  \frac{1+tc_{EH}(K)^{\frac{1}{2}}}{\left(\frac{\V(B^{2n} + tK)}{\V(B^{2n})}\right)^{\frac{1}{2n}}}.\end{equation*}

\begin{lemma}\label{tilde} The function $\widetilde{F}$ has the following properties:
\begin{itemize}
    \item[($\widetilde{\mathrm{F}}1$)] $\widetilde{F}(0)=1$,
    \item[($\widetilde{\mathrm{F}}2$)] $\widetilde{F}$ is $C^{\infty}$-smooth,
    \item[($\widetilde{\mathrm{F}}3$)]  $\widetilde{F}(t) \leq F(t) \text{  for all  } t \geq 0.$   
\end{itemize}
\end{lemma}

\begin{proof} The first property is clear, the second follows from \eqref{steiner}, and the third follows from the symplectic Brunn-Minkowski inequality, \eqref{bm}. \end{proof} 

Property ($\widetilde{\mathrm{F}}2$) ensures that the right derivative of $\widetilde{F}$ at $t=0$ exists. Properties (F1), (F3), ($\widetilde{\mathrm{F}}1$) and ($\widetilde{\mathrm{F}}3$) imply that this derivative must be nonpositive, i.e.,
\begin{equation}\label{need}
    \widetilde{F}_+'(0) \leq 0.
\end{equation}
A straight forward computation, together with \eqref{quer}, then reveals that \eqref{need} is equivalent to the desired inequality,
$$c_{EH}(K)^{\frac{1}{2}} \leq \frac{\mw(K)}{2}.$$

\begin{remark}\label{qm}
If one defines the normalized $i$th quermassintegral of $K$ to be
$$\overline{W}_i(K)=\left(\frac{W_{i}(K)}{W_{i}(B^{2n})}\right)^{\frac{1}{2n-i}},$$ then Viterbo's inequality is \begin{align}\label{vith}c_{EH}(K) \leq (\overline{W}_0(K))^2\end{align} and the inequality of Artstein-Avidan and Ostrover is \begin{align}\label{aoh}c_{EH}(K) \leq (\overline{W}_{2n-1}(K))^2.\end{align} The Alexsandrov-Fenchel inequality for mixed volumes implies that
\begin{align*}
  \overline{W}_{i-1}(K) \leq \overline{W}_i(K) 
\end{align*}
for all $i=1, \dots,2n-1.$ 
Given that \eqref{vith} is now known to be false, one might ask the following.
\begin{question} Is there an $i \in [1,2n-2]$ such that $c_{EH}(K) \leq (\overline{W}_{i}(K))^2$ for all convex bodies $K$ in $\RR^{2n}$?
\end{question}

The most ambitious version of this question is to ask whether $$c_{EH}(K) \leq \left( \frac{\mathrm{Surface\, Area}(K)}{\mathrm{Surface\, Area}(B)}\right)^{\frac{2}{2n-1}}$$
holds for every convex body in $\RR^{2n}.$

\end{remark}

\section{Toric symmetry and convex bodies with minimal mean width}\label{nontoric}

Inequality \eqref{ao} relates a symplectic invariant of a convex body $K$ to its mean width. Since the mean width is not a symplectic invariant one can refine \eqref{ao} by minimizing over the mean width of the images of $K$ under various classes of symplectomorphisms. For example, given the measurements 
\begin{equation*}
  \lmw(K) \eqdef \inf_{P \in \mathrm{Sp}
(2n)} \mw (PK),
\end{equation*}
and 
\begin{equation*}
 \nlmw(K) \eqdef \inf_{\phi \in \Symp} \mw (\phi (K)).
\end{equation*}
we have the refinements
$$
c_{EH}(K) \leq \frac{\nlmw(K)^2}{4}  \leq \frac{\lmw(K)^2}{4}  \leq \frac{\mw(K)^2}{4}.
$$
The following two optimization problems are crucial to understanding the second and third of these inequalities.\footnote{A convex body $K$ satisfying \eqref{lin} or \eqref{nonlin} will be referred to as a {\em solution} of the corresponding problem.} 
\begin{problem}\label{pr1}
Characterize the convex bodies $K \subset \RR^{2n}$ that satisfy \begin{equation}\label{lin}\lmw(K) = \mw(K).\end{equation}
\end{problem}
 
\begin{problem}\label{pr2}
Characterize the convex bodies $K \subset \RR^{2n}$ that satisfy \begin{equation}\label{nonlin}\nlmw(K) = \mw(K).\end{equation}
\end{problem}


In this section  we state and test a conjectured solution to Problem \ref{pr2}. We start by recalling the relevant background and supporting evidence. 

In dimension two, the solutions to Problem \ref{pr1} and \ref{pr2} are both known. The work of Green in \cite{green} yields a complete solution to Problem \ref{pr1}.
\begin{theorem}[Green, \cite{green}]\label{green}
A convex body $K \subset \RR^2$ satisfies $\lmw(K) = \mw(K)$
if and only if $$\int_0^{2\pi} h_{K}(\theta)\cos 2 \theta \, d \theta =0=\int_0^{2\pi} h_{K}(\theta)\sin  2 \theta \, d \theta.$$
\end{theorem}
\noindent On the other hand, Problem \ref{pr2} is trivial in dimension two, as a simple argument reveals that $K \subset \RR^2$ satisfies $\nlmw(K) = \mw(K)$ if and only if $K= B^2$.

In higher dimensions much less is known. Several results in this direction are proved in \cite{ak}. 
\begin{definition} A subset $X \subset \RR^{2n}$ is said to be {\em toric} if it is invariant under the standard Hamiltonian action of the $n$-dimensional torus on $\RR^{2n} =\CC^n$ given by
$$(\theta_1, \dots, \theta_n) \cdot (z_1, \dots,z_n) = \left(e^{i\theta_1}z_1, \dots, e^{i \theta_n} z_n\right).$$
\end{definition}

\begin{remark}\label{complex}
This definition is tied to a fixed complex basis of $\CC^n$.
\end{remark}

The following result settles the infinitesimal (local) version of  Problem \ref{pr1}. 

\begin{theorem}[Theorem 1.4 of \cite{ak}]\label{toric}

If $QK$ is toric for some unitary matrix $Q \in \mathrm{U}(n)$, then the identity matrix  $\II \in \mathrm{Sp}(2n)$ is a local minimum of the function $\mathrm{Sp}(2n) \to \RR$ defined by 
\begin{align*}
 P \mapsto \mw(PK).
\end{align*}
Moreover, if $K$ is a toric convex domain which is strictly convex and has a $C^2$-smooth boundary, then $\II \in \mathrm{Sp}(2n)$ is an isolated local minimum.
 
\end{theorem}

Being toric, up to the action of the unitary group, is not a necessary condition for a convex body $K \subset \RR^{2}$ to (locally) minimize the mean width. In dimension two, this follows immediately from Theorem \ref{green}. A simple example that illustrates this fact in dimension four is the Lagrangian bidisk.

\begin{proposition}[Proposition 1.12 of \cite{ak}]
For the Lagrangian bidisk
$$\mathbf{P}_L = \left\{ (x_1, x_{2}, y_1, y_{2}) \in \mathbb{R}^{4} \bigmid  x_1^2 +x_2^2 \leq 1,\, y_1^2 +y_2^2 \leq 1\right\},$$ the identity matrix is an isolated  local minimum of the function  $\mathrm{Sp}(4) \to \RR$ defined by
\begin{align*}
 P \mapsto \mw(P \mathbf{P}_L).
\end{align*} 
Moreover, $\mathbf{P}_L$ is not equal to $QK$ for any toric convex body $K$ and any matrix $Q \in \mathrm{U}(2)$.
\end{proposition}

The Lagrangian bidisk also exemplifies differences between the solution sets of Problems \ref{pr1} and \ref{pr2} that are not visible in dimension two. In particular, in Proposition 1.14 of \cite{ak}, it is shown that the toric embedding of the interior of $\mathbf{P}_L$ from \cite{vr} can be used to prove that $$\nlmw(\mathbf{P}_L) < \mw(\mathbf{P}_L).$$

These results support the following conjectural solution to Problem \ref{pr2}.

\begin{conjecture}\label{conj}
 A convex body $K \subset \RR^{2n}$ satisfies 
 $\nlmw(K) =\mw (K)$
if and only if $QK$ is toric for some unitary matrix $Q \in \mathrm{U}(n)$.
\end{conjecture}

This is further supported by the following stability result.

\begin{theorem}[Theorem 1.3 of \cite{ak}]\label{variation}
Let $K$ be a convex body in $\RR^{2n}$. If  $QK$ is toric for some unitary matrix $Q \in \mathrm{U}(n)$, then for any smooth path $\phi^t$ in $\Symp$ that passes through the identity map at $t=0$, we have
\begin{align*}
    \left.\frac{d}{dt}\right|_{t=0} \mw(\phi^t(K)) = 0.
\end{align*}
\end{theorem}

\subsubsection{New local minimizers} In this section we identify a large family of elementary convex bodies $K$ in $\RR^{2n}$ for which the identity matrix is an isolated  local minimum of the function  $\mathrm{Sp}(2n) \to \RR$ defined by
$
 P \mapsto \mw(P K).$

Fix a symplectic basis  of $\RR^{2n}$, $$\{e_1, \dots, e_n, f_1, \dots, f_n\},$$ such that $\mathrm{Span}\{e_j,f_j\}$ corresponds to the $j$th complex subspace determined by the definition of being toric (see Remark \ref{complex}). Consider subsets of $\RR^{2n}$ of the form 
\begin{equation*}\label{balls} \rho_1B^{m_1} \times \rho_2 B^{m_2} \times \cdots \times \rho_k B^{m_k},\end{equation*}
where $\rho_\ell B^{m_\ell}$ is the ball of radius $\rho_\ell$ in one of the coordinate subspaces of $\RR^{2n}$ of dimension $m_\ell$, and $\sum_{\ell=1}^k m_{\ell} =2n$. Such a product is determined uniquely by the $k$-tuple of positive numbers $\bar{\rho}=(\rho_1, \dots, \rho_k)$, and two partitions of $\{1, \dots, n\}$, $$\mathcal{I} =\{I_1, \dots, I_k\} \text{  and  }\mathcal{J}=\{J_1, \dots, J_k\}$$  where for each $\ell$ we allow for the possibility that one of, but not both of,  $I_\ell$ and $J_\ell$ is empty. In particular, $m_\ell = |I_{\ell}|+|J_{\ell}|$ and, 
setting $I_\ell= \{i(\ell,1), \dots, i(\ell,|I_\ell|)\}$ and  $J_\ell= \{j(\ell,1), \dots, j(\ell,|J_\ell|)\}$, $B^{m_\ell}$ is the unit ball in the subspace
$$\mathrm{Span}\{   e_{i(\ell,1)},\dots e_{i(\ell, |I_\ell|)},f_{j(\ell,1)}, \dots, f_{j(\ell, |J_\ell|)}   \}.$$
With this correspondence understood, we will write
$$\mathcal{P}(\bar{\rho},\mathcal{I},\mathcal{J}) = \rho_1B^{m_1} \times \rho_2 B^{m_2} \times \cdots \times \rho_k B^{m_k}.$$

\begin{theorem}\label{min}
The identity matrix is an isolated  local minimum of the function  $\mathrm{Sp}(2n) \to \RR$ defined by
$
 P \mapsto \mw(P \mathcal{P}(\bar{\rho},\mathcal{I},\mathcal{J}))
 $
if
\begin{equation}\label{cond}
      I_\ell \cap J_{\ell'} \neq \emptyset  \implies |I_\ell| + |J_\ell| = |I_{\ell'}| + |J_{\ell'}| \text{ and  } \rho_\ell=\rho_{\ell'} ,\quad \text{for all $\ell,\, \ell'$.}
    \end{equation}
\end{theorem}



\begin{proof} Since each symplectic matrix has a unique polar decomposition and the mean width is invariant under the action of the orthogonal group, it suffices to show that the identity matrix $\II$ is an isolated local minimum of the function $ Sym^+(\mathrm{Sp}(2n)) \to \RR$ defined by $S \mapsto\mw(S \mathcal{P}(\bar{\rho},\mathcal{I},\mathcal{J}))$ where $Sym^+(\mathrm{Sp}(2n))$ is the set of positive definite, symmetric, symplectic matrices.

Fixing a smooth path $S(s)$ in $Sym^+(\mathrm{Sp}(2n))$ with $S(0)=\II$ and $S'(0) \neq 0$, it suffices to show that the function
    \begin{align}\label{eq:calc}
        f(s) = \int_{S^{2n-1}} h_{S(s)\mathcal{P}(\bar{\rho},\mathcal{I},\mathcal{J})}(u) \, d\sigma
    \end{align}
    satisfies $f'(0) =0$ and $f''(0) >0.$

Let $\pi_\ell$ be the orthogonal projection from $\RR^{2n}$  to the subspace $\mathrm{Span}(B^{m_\ell})$. The support function of $\mathcal{P}(\bar{\rho},\mathcal{I},\mathcal{J})$ is given by 
    \begin{equation*}
        h_{\mathcal{P}(\bar{\rho},\mathcal{I},\mathcal{J})}(u) = \sum_{\ell=1}^k \rho_\ell \| \pi_{\ell}u \|.
    \end{equation*}
The homogeneous extension of this support function is defined by 
\begin{align*}
    H_{\mathcal{P}(\bar{\rho},\mathcal{I},\mathcal{J})}(z)  = h_{\mathcal{P}(\bar{\rho},\mathcal{I},\mathcal{J})} \left( z/ \|z\|\right).
\end{align*}
and we note that the gradient $$\nabla H_{\mathcal{P}(\bar{\rho},\mathcal{I},\mathcal{J})}(z) = \sum_{\ell=1}^k \frac{\rho_\ell  \pi_\ell z}{\| \pi_\ell z\|}$$ is defined almost everywhere. 
For every  $S \in Sym^+(\mathrm{Sp}(2n))$ we have \begin{align}\label{eq;trans}
    h_{S{\mathcal{P}(\bar{\rho},\mathcal{I},\mathcal{J})}}(u) = H_{\mathcal{P}(\bar{\rho},\mathcal{I},\mathcal{J})}(S u).
\end{align}
Hence, $f'(0)=0$ if and only if 
\begin{align}\label{eq:zero}
    \int_{S^{2n-1}} \langle \nabla H_{\mathcal{P}(\bar{\rho},\mathcal{I},\mathcal{J})}(u), Y u \rangle \, d\sigma =0
\end{align}
for any matrix $Y$ of the form  
\begin{align*}
    Y = \begin{pmatrix}
    C & D\\
    D & -C
    \end{pmatrix}, 
\end{align*}
where the submatrices $C$ and $D$ are both symmetric. 

From the expression above, we have
\begin{equation}\label{target}
   \int_{S^{2n-1}} \langle \nabla H_{\mathcal{P}(\bar{\rho},\mathcal{I},\mathcal{J})}(u), Y u \rangle \, d\sigma = \sum_{\ell=1}^k \int_{S^{2n-1}} \frac{\rho_\ell \langle \pi_\ell u, \pi_\ell Y u \rangle }{\| \pi_\ell u\|} \, \, d \sigma.
    \end{equation}
A simple symmetry argument implies that all {\em cross-terms} that appear in \eqref{target} vanish. In particular,  for all $\ell =1, \dots k$, we have 
    \begin{equation*}
        \int_{S^{2n-1}} \frac{x_i y_j}{\|\pi_\ell u\|} \, \, d \sigma = 0\,\,\text{for all $i$ and $j$},
    \end{equation*}
    and 
    \begin{equation*}
        \int_{S^{2n-1}} \frac{x_i x_j}{\|\pi_\ell u\|} \, \, d \sigma = \int_{S^{2n-1}} \frac{y_i y_j}{\|\pi_\ell u\|} \, \, d \sigma =0\,\,\text{for all $i \neq j$}.
    \end{equation*}
The remaining terms then yield the expression
    \begin{align}\label{fi}
    f'(0)=\sum_{\ell=1}^k\left(\sum_{m=1}^{|I_{\ell}|} \int_{S^{2n-1}} \frac{ \rho_\ell c_{i(\ell, m), i(\ell, m)} x_{i(\ell, m)}^2}{\|\pi_\ell u\|} \, \, d \sigma - \sum_{m = 1}^{|J_\ell|} \int_{S^{2n-1}} \frac{ \rho_\ell c_{j(\ell, m), j(\ell, m)}y_{j(\ell, m)}^2}{\|\pi_\ell u\|} \, \, d \sigma \right).
    \end{align}
    
For each integer $r$ between $1$ and $n$ there will be exactly one term in \eqref{fi} with $x_r^2$ in the numerator of the integrand, and another with $y_r^2$.  To show  that $f'(0)=0$ it remains to show that condition \eqref{cond} implies that these terms cancel.

Fixing $r$, we note that it is contained in a unique $I_\ell$ and a unique $J_{\ell'}$. If $\ell = \ell'$, then the terms in \eqref{fi} corresponding to $x_r^2$  and $y_r^2$ are $$\int_{S^{2n-1}} \frac{ \rho_{\ell}  c_{r, r}x_{r}^2}{\|\pi_\ell u\|} \, \, d \sigma \quad \quad \text{  and  } \quad \quad -\int_{S^{2n-1}} \frac{ \rho_{\ell}  c_{r, r}y_{r}^2}{\|\pi_\ell u\|} \, \, d \sigma,$$
which clearly cancel. If $\ell \neq \ell'$, then the terms in \eqref{fi} corresponding $x_r^2$  and $y_r^2$ are \begin{equation}\label{r}\int_{S^{2n-1}} \frac{ \rho_\ell c_{r, r}x_{r}^2}{\|\pi_\ell u\|} \, \, d \sigma \quad \quad \text{  and  } \quad \quad -\int_{S^{2n-1}} \frac{ \rho_{\ell'} c_{r, r}y_{r}^2}{\|\pi_{\ell'} u\|} \, \, d \sigma. \end{equation}
In this case, \eqref{cond} implies that $|I_{\ell}| + |J_{\ell}| = |I_{\ell'}| + |J_{\ell'}|$ and $\rho_{\ell}=\rho_{\ell'}$.
   Setting $\mathcal{X}_{\ell} = \{x_{(i,1)}, \dots, x_{(i, |I_{\ell}|)}\}$ and $\mathcal{Y}_{\ell} = \{y_{(j,1)}, \dots, y_{(j, |J_{\ell}|)}\}$, it follows that we can choose a bijection from $ \mathcal{X}_{\ell} \cup \mathcal{Y}_{\ell} $ to $\mathcal{X}_{\ell'} \cup \mathcal{Y}_{\ell'}$ that maps $x_r$ to $y_r$. Extending this bijection to an orientation preserving bijection of $\{x_1, \dots, x_n, y_1, \dots, y_n\}$, the corresponding  change of variables yields \begin{equation*}
        \int_{S^{2n-1}} \frac{x_r^2}{\|\pi_\ell u\|} \, \, d \sigma = \int_{S^{2n-1}} \frac{y_r^2}{\|\pi_{\ell'}u\|} \, \, d \sigma,
    \end{equation*}
    and so the terms in \eqref{r} cancel, as desired.

\bigskip

It remains to prove that $f''(0)>0$. Note that   
\begin{align*}
    f''(s) = \int_{S^{2n-1}} \text{Hess}( H_{\mathcal{P}(\bar{\rho},\mathcal{I},\mathcal{J})})(S(s)u)(S'(s)u, S'(s)u) \, d\sigma + \int_{S^{2n-1}} \langle \nabla  H_{\mathcal{P}(\bar{\rho},\mathcal{I},\mathcal{J})}(S(s)u), S''(s)u \rangle \,d\sigma, 
\end{align*}
where, $\text{Hess}( H_{\mathcal{P}(\bar{\rho},\mathcal{I},\mathcal{J})})$ denotes the Hessian of $ H_{\mathcal{P}(\bar{\rho},\mathcal{I},\mathcal{J})}$ which is well-defined  and positive definite almost everywhere. Since $S'(0) \neq 0$, the first summand on the right is positive when $s=0$, and it suffices to prove that 
\begin{align*}
\int_{S^{2n-1}} \langle \nabla  H_{\mathcal{P}(\bar{\rho},\mathcal{I},\mathcal{J})}(u), S''(0)u \rangle \,d\sigma \geq 0.
\end{align*}

We can write $S'(s) = S(s)X(s)$ for a family of symmetric matrices $X(s) \in \mathfrak{sp}(2n)$ of the form 
\begin{align}\label{eq:X}
    X(s) = \begin{pmatrix}
    C(s) & D(s)\\
    D(s) & -C(s)
    \end{pmatrix},
\end{align}
where the submatrices $C(s)$ and $D(s)$ are all symmetric. We then have $S''(s) = S(s) (X(s)^2 + X'(s))$ and $S''(0) = X(0)^2 + X'(0)$. Arguing as above, we have 
\begin{align*}
    \int_{S^{2n-1}} \langle \nabla  H_{\mathcal{P}(\bar{\rho},\mathcal{I},\mathcal{J})}(u), X'(0)u \rangle \, d\sigma = 0,
\end{align*}
since $X'(0)$ has the same form as $Y$. It remains to show that 
$
    \int_{S^{2n-1}} \langle \nabla  H_{\mathcal{P}(\bar{\rho},\mathcal{I},\mathcal{J})}(u), X(0)^2u \rangle d\sigma.
$ is nonnegative.
By \eqref{eq:X}  we have 
\begin{align*}
    X(0)^2 =&\begin{pmatrix}
    C(0)^2+D(0)^2 & C(0)D(0) - D(0)C(0)\\
    D(0)C(0) - C(0)D(0) & C(0)^2 +D(0)^2
    \end{pmatrix}=:    \begin{pmatrix}
    L & M\\
    -M & L
    \end{pmatrix},
\end{align*}
where the submatrix $L =(l_{ij})$ is symmetric, the submatrix $M=(m_{ij})$ is skew-symmetric and the diagonal entries of $L$ are all nonnegative. 
Noting that the relevant  cross-terms again vanish, it follows that $
    \int_{S^{2n-1}} \langle \nabla  H_{\mathcal{P}(\bar{\rho},\mathcal{I},\mathcal{J})}(u), X(0)^2u \rangle d\sigma.
$ is equal to 
\begin{align*}
    \sum_{\ell=1}^k\rho_\ell\left(\sum_{m=1}^{|I_{\ell}|} \int_{S^{2n-1}} \frac{ l_{i(\ell, m), i(\ell, m)} x_{i(\ell, m)}^2}{\|\pi_\ell u\|} \, \, d \sigma + \sum_{m = 1}^{|J_\ell|} \int_{S^{2n-1}} \frac{ l_{j(\ell, m), j(\ell, m)}y_{j(\ell, m)}^2}{\|\pi_\ell u\|} \, \, d \sigma \right) 
    \end{align*}
which is nonnegative, as desired

\end{proof}
\begin{remark}  A simple argument reveals that the sufficient condition of Theorem \ref{min} can be augmented to a necessary and sufficient condition of the following form  
\begin{align*}
I_\ell \cap J_{\ell'} \neq \emptyset    \implies&    |I_\ell| + |J_\ell| = |I_{\ell'}| + |J_{\ell'}| \text{ and  } \rho_\ell=\rho_{\ell'}, \\
&\text{or $\bar{\rho}$ belongs to a proper subspace of $\RR^{n}$ determined by $\mathcal{I}$ and $\mathcal{J}$}
\end{align*}
We will not explore this detail further here. 
\end{remark}

\subsubsection{Testing Conjecture \ref{conj}.} One can also determine exactly when there exists a $Q \in \mathrm{U}(n)$ such that $$Q(\rho_1B^{m_1} \times \rho_2 B^{m_2} \times \cdots \times \rho_k B^{m_k})$$ is toric.

\begin{lemma}\label{non}
The set $Q(\rho_1B^{m_1} \times \rho_2 B^{m_2} \times \cdots \times \rho_k B^{m_k})$ is toric for some $Q \in \mathrm{U}(n)$ if and only if every factor of $\rho_1B^{m_1} \times \rho_2 B^{m_2} \times \cdots \times \rho_k B^{m_k}$ is symplectic.
\end{lemma}

\begin{proof}
One can easily check that the product $\rho_1B^{m_1} \times \rho_2 B^{m_2} \times \cdots \times \rho_k B^{m_k}$ is toric if and only each of its factors is symplectic.  For  $Q \in \mathrm{U}(n)$, the set $$Q(\rho_1B^{m_1} \times \rho_2 B^{m_2} \times \cdots \times \rho_k B^{m_k})$$
is a standard product of balls with respect to the symplectic basis $$\{Q(e_1), \dots, Q(e_n), Q(f_1), \dots, Q(f_n)\}.$$ By our choice of the basis $\{e_1, \dots f_n\}$, the set $$Q(\rho_1B^{m_1} \times \rho_2 B^{m_2} \times \cdots \times \rho_k B^{m_k})$$ is then toric if and only if there is a permutation $\sigma \in S_n$ such that $$\mathrm{Span}\{Q(e_i), Q(f_i)\} = \mathrm{Span}\{e_{\sigma(i)}, f_{\sigma(i)})\}$$ for $i=1, \dots,n$. By the discussion above, this implies that $Q(\rho_1B^{m_1} \times \rho_2 B^{m_2} \times \cdots \times \rho_k B^{m_k})$
is toric if and only each of its factors is symplectic. Since the factors of $Q(\rho_1B^{m_1} \times \rho_2 B^{m_2} \times \cdots \times \rho_k B^{m_k})$ are symplecic if and only if those of $\rho_1B^{m_1} \times \rho_2 B^{m_2} \times \cdots \times \rho_k B^{m_k}$ are, we are done.
\end{proof}

We now have a large set of examples on which Conjecture \ref{conj} can be tested. Let $\mathcal{P}^{2n}_{test}$ be the collection of all products $$\mathcal{P}(\bar{\rho},\mathcal{I},\mathcal{J}) = \rho_1B^{m_1} \times \rho_2 B^{m_2} \times \cdots \times \rho_k B^{m_k}$$
such that $(\bar{\rho},\mathcal{I},\mathcal{J})$   satisfies condition \eqref{cond} and there is at least one  nonsymplectic factor. By Theorem \ref{min}, for every $K_{test} \in  \mathcal{P}^{2n}_{test}$  the identity matrix is an isolated  local minimum of the function  $\mathrm{Sp}(2n) \to \RR$ defined by
$
 P \mapsto \mw(P K_{test}).$
On the other hand, Lemma \ref{non} implies that  $QK_{test}$ is not toric for any $Q \in \mathrm{U}(n)$. Conjecture \ref{conj} restricts to this family as the following simple assertion.

\begin{claim}\label{test}
    For any $K_{test}$ in $\mathcal{P}^{2n}_{test}$ there is a symplectomorphism $\Phi$ such that  $\mw(\Phi(K_{test})) < \mw(K_{test})$.
\end{claim}

 As mentioned above, if $K_{test}$ is the Lagrangian bidisk $\mathbf{P}_L$, then a symplectomorphism as in Claim \ref{test} can be obtained from the embedding of the interior of $\mathbf{P}_L$ from \cite{vr}, \cite{ak}. The remarkable embedding from \cite{vr} is constructed for other purposes, and it is not clear that a general procedure for constructing the desired symplectomorphisms of Conjecture \ref{test} can be derived in this manner. Below, we describe symplectomorphisms that settle Claim \ref{test} in some cases. The construction is based on a simple model and inspired by the general formula for the mean width of products.

\medskip

\noindent{\bf The Model.} The square $\square \eqdef [-1,1] \times [-1,1]$ is an element of $\mathcal{P}^{2}_{test}$. A simple computation yields
$$
  \mw(\square) =\frac{8}{\pi}.
$$
The ball $\frac{2}{\sqrt{\pi}} B^2$ has the same volume as $\square$ and $$\mw \left( \frac{2}{\sqrt{\pi}} B^2\right)=\frac{4}{\sqrt{\pi}}.$$

\begin{lemma}\label{roll} For any $\epsilon>0$, there is  a symplectomorphism $\phi_{\epsilon}$ of $\RR^2$ such that  $$\mw(\phi_{\epsilon}(\square)) \leq  \frac{4}{\sqrt\pi} +\epsilon.$$\end{lemma}

\begin{proof}

Let $g_{\epsilon} \colon \square \to \RR^2$ be a smooth embedding with the following properties:
\begin{enumerate}
    \item[(i)] $g_{\epsilon}(\square)$ is convex,
    \item[(ii)] $\V(g_{\epsilon}(\square)) = \V(\square) = 4$,
    \item[(iii)] The Hausdorff distance between $g_{\epsilon}(\square)$ and $\frac{2}{\sqrt{\pi}} B^2$ is less than $\epsilon.$
\end{enumerate}
By property (ii), and the version of Moser's theorem for manifolds with corners established in \cite{bmpr}, the map $g_{\epsilon}$ yields a symplectic embedding $f_{\epsilon} \colon \square \to \RR^2$ with the same image as $g_{\epsilon}$ and hence properties (i)-(iii). Extend this map $f_{\epsilon}$ to an embedding $G_{\epsilon} \colon \mathcal{N}(\square) \to \RR^2$ where $\mathcal{N}(\square)$ is a convex neighborhood of $\square$ with a smooth boundary and $\V(G_{\epsilon} (\mathcal{N}(\square))) = \V(\mathcal{N}(\square))$. Applying the relative version of Moser's theorem for manifolds with boundary, established in \cite{ban}, relative to $\square$, we get a symplectic embedding $$F_{\epsilon} \colon \mathcal{N}(\square) \to \RR^2 $$ such that $F_{\epsilon}(\square) =f_{\epsilon}(\square) = g_{\epsilon}(\square)$ is convex and within $\epsilon$ of  $\frac{2}{\sqrt{\pi}} B^2$ with respect to the Hausdorff metric. Finally, applying the {\em Extension after Restriction Principle} (\cite{schl}),  to $F_{\epsilon}$ relative to $\square$, we get the desired symplectomorphism $\phi_{\epsilon}$.
\end{proof}


We will apply the model above to elements of $\mathcal{P}^{2n}_{test}$ that have factors that form a symplectic square. The resulting effect on the mean width will be determined by the following result.

\begin{lemma}[see formula (13.51) in \cite{sant}]\label{product}
    Suppose that $K_1$ is a convex body in $\RR^{d_1}$,  $K_2$ is a convex body in $\RR^{d_2}$, and $d_1,d_2 \geq 2$. Then the mean width of the convex body $K_1 \times K_2 \subset \RR^{d_1 + d_2}$ is given by 
\begin{equation*}
    \mw(K_1 \times K_2) = \frac{1}{2(d_1+d_2)} \frac{\kappa_{d_1+d_2-1}}{\kappa_{d_1+d_2}} \left( d_1 \frac{\kappa_{d_1}}{\kappa_{d_1-1}}\mw(K_1) + d_2 \frac{\kappa_{d_2}}{\kappa_{d_2-1}} \mw(K_2)  \right),
\end{equation*}
 where $\kappa_d =\V(B^d).$   
\end{lemma}

This logarithmic property of the mean width implies that if $\psi_1$ is a diffeomorphism of $\RR^{d_1}$ such that $\psi_1(K_1)$ is convex and $\mw(\psi_1(K_1))< \mw(K_1)$, then the diffeomorphism of $\RR^{d_1 + d_2}= \RR^{d_1} \times \RR^{d_2}$ defined by $(v_1, v_2) \mapsto (\psi_1(v_1), v_2)$ maps $K_1 \times K_2$ to a convex set with a smaller mean width.

With this, we now address Claim \ref{test}.

\begin{proposition}
Suppose that  $K_{test}$ is an element of $\mathcal{P}^{2n}_{test}$. If $K_{test}$ has a one dimension factor, then there is a symplectomorphism $\Phi$ such that $\Phi(K_{test})$ is convex and $\mw(\Phi(K_{test})) < \mw(K_{test})$.
\end{proposition}
  
\begin{proof}
Relabelling, if necessary, we may assume that $$K_{test} = \rho_1 I_{e_1} \times \rho_2 B^{m_2} \times \cdots \times \rho_k B^{m_k},$$ where $I_{e_j} = \{te_j \mid t \in [-1,1]\}$. By condition \eqref{cond}, the $J_{\ell}$ that contains $1$ only contains $1$, $I_{\ell}$ is empty, and $\rho_{\ell} = \rho_1$. Hence, we may assume that
$$K_{test} = \rho_1( I_{e_1} \times  I_{f_1}) \times \rho_3 B^{m_2} \times \cdots \times \rho_k B^{m_k},$$ where $I_{f_j}= \{tf_j \mid t \in [-1,1]\}$. For $0< \epsilon < \frac{8}{\pi} - \frac{4}{\sqrt{\pi}}$, let $\phi_{\epsilon}$ be the symplectomorphsm from Lemma \ref{roll}. Define $\Phi \colon \RR^{2n} \to \RR^{2n}$ by
$$\Phi(x_1,y_1, x_2, y_2 \dots, x_n, y_n) = \left(\rho \phi_{\epsilon}\left(\frac{x_1}{\rho}, \frac{y_1}{\rho}\right), x_2, y_2, \dots, x_n,y_n\right).$$
By the construction of $\phi_{\epsilon}$, the map $\Phi$ is a symplectomorphism, the image $\Phi(K_{test})$ is convex, and, by Lemma \ref{product}, we have  $\mw(\Phi(K_{test}))< \mw(K_{test}).$
\end{proof}

\begin{remark} The following naive extension of the ideas above has not yet been validated. Consider a $K_{test}$ in $\mathcal{P}^{2n}_{test}$ and any nonsymplectic factor $\rho_{\ell} B^{m_{\ell}}$ of $K_{test}$. Either there is an  $i$ in $I_{\ell}$ that is not in $J_{\ell}$, or a $j$ in $J_{\ell}$ that is not in $I_{\ell}$. Assume the former.
Then, by condition \eqref{cond}, there is another nonsymplectic factor $\rho_{\ell'} B^{m_{\ell'}}$ of $K_{test}$ such that  $\rho_{\ell'} = \rho_{\ell}$ and $m_{\ell'} = m_{\ell'}$ and $i$ is in $J_{\ell'}$. It follows from this that the orthogonal projection of $K_{test}$ to the $x_{i}y_{i}$-plane is equal to the square $\rho_{\ell} \square.$ One might then ask whether the image of $K_{test}$ under the symplectomorphism 
$$(x_1,y_1,\dots,x_i, y_i, \dots x_n, y_n) \mapsto \left(x_1, y_1, \dots, \rho_\ell \phi_{\epsilon}\left(\frac{x_i}{\rho_\ell}, \frac{y_i}{\rho_\ell}\right), \dots, x_n,y_n\right).$$
has strictly smaller mean width than $K_{test}$. An affirmative answer would settle Claim \ref{test}.
\end{remark}


\begin{thebibliography} {99}


\bibitem{ab} A. Abbondandolo, G. Benedetti, On the local systolic optimality of Zoll contact forms, {\em Geom. Funct. Anal.}, 33 (2023), 299--363.


\bibitem{ak} J. Ahn, E. Kerman, Symplectic variations of convex bodies and the mean width, preprint arXiv:2311.02870.





\bibitem{ao} S. Artstein-Avidan, Y. Ostrover, A Brunn-Minkowski inequality for symplectic capacities of convex domains, {\em IMRN}, 13 (2008).






\bibitem{ban} A. Banyaga, Formes-volume sur les vari\'{e}t\'{e}s \`{a} bord, {\em Enseignement Math.} (2) 20 (1974), 127--131.




\bibitem{bj} I. Belegradek, Z. Jiang,
Smoothness of Minkowski sum and generic rotations,
{\em Journal of Mathematical Analysis and Applications},
Volume 450, Issue 2,
(2017), 1229--1244.





\bibitem{bmpr} M. Bruveris, P. Michor, A. Parusinski, A. Rainer, Armin, Moser's theorem on manifolds with corners. {\em Proceedings of the American Mathematical Society}, 146 (2016), 4889--4897. 














\bibitem{eh1} I. Ekeland and H. Hofer, Symplectic topology and
Hamiltonian dynamics, {\it Math. Z.}, \textbf{200} (1989), 355--378.















\bibitem{green} J.W. Green, Length and area of a convex curve under affine transformation, {\it Pacific J. Math.} \textbf{3} (1953), 393--402. 



























\bibitem{kl1} E. Kerman and Y. Luang,  Higher index symplectic capacities do not satisfy the symplectic Brunn-Minkowski inequality, {\em Isr. J. Math.} 245, (2021), 27--38.  

\bibitem{kis} C. O. Kiselman, How smooth is the shadow of a smooth convex body?, {\em J. London Math. Soc.}, (2) 33 (1986), 101--109.

\bibitem{hko} P. Haim-Kislev, Y. Ostrover, A Counterexample to Viterbo's Conjecture, arXiv:2405.16513.





















\bibitem{vr} V.G.B. Ramos, Symplectic embeddings and the Lagrangian bidisk, {\em Duke Math. J.}, 166 (2017), 1703--1738.





\bibitem{sant} L. Santal\'o, Integral Geometry and Geometric Probability. Second Edition. Cambridge Mathematical Library. Cambridge University Press, 2004.

\bibitem{schl} F. Schlenk, Embedding problems in symplectic geometry
De Gruyter Expositions in Mathematics 40. Walter de Gruyter Verlag, Berlin. 2005.

\bibitem{schn} R. Schneider,  (2013). Convex Bodies: The Brunn-Minkowski Theory (Encyclopedia of Mathematics and its Applications). Cambridge: Cambridge University Press. doi:10.1017/CBO9781139003858.



\bibitem{sv} Y. Shenfeld, R. van Handel, Mixed volumes and the Bochner method, {\em Proceedings of the American Mathematical Society}, 147 (2019), 5385--5402.









\bibitem{vit} C. Viterbo, Metric and Isoperimetric problems in symplectic geometry, Journal of the AMS, \textbf{13} (2000), Pages 411--431.



\end{thebibliography}
\end{document}